\documentclass[5p,times]{elsarticle}
\usepackage{amssymb}
\usepackage{amsthm}
\usepackage{makecell}
\usepackage{amsmath}
\usepackage{color,xcolor}
\newtheorem{remark}{Remark}
\usepackage[utf8]{inputenc}
\usepackage{booktabs, caption, makecell}

\usepackage{threeparttable}
\usepackage{epigraph}
\usepackage{enumitem}
\usepackage{url}
\usepackage{color}
\usepackage{bbm}
\usepackage{amsmath, amssymb, amsfonts, amsthm}
\usepackage{siunitx}
\usepackage[linesnumbered, lined, boxed]{algorithm2e}
\usepackage{multirow}
\usepackage[figuresright]{rotating}
\usepackage{amsmath}

\usepackage{tablefootnote}

\DeclareMathOperator*{\minimize}{minimize}
\DeclareMathOperator*{\maximize}{maximize}
\DeclareMathOperator*{\subt}{subject \ to}
\newtheorem{theorem}{Theorem}[section]

\begin{document}

\begin{frontmatter}

\title{Computationally efficient solution of mixed integer model predictive control problems via machine learning aided Benders Decomposition}

\author{Ilias Mitrai}
\author{Prodromos Daoutidis\fnref{label2}}
\fntext[label2]{Corresponding author}
\ead{daout001@umn.edu}
\address{Department of Chemical Engineering and Materials Science, University of Minnesota, Minneapolis, MN 55455}

\begin{abstract} 
Mixed integer Model Predictive Control (MPC) problems arise in the operation of systems where discrete and continuous decisions must be taken simultaneously to compensate for disturbances. The efficient solution of mixed integer MPC problems requires the computationally efficient and robust online solution of mixed integer optimization problems, which are generally difficult to solve.  In this paper, we propose a machine learning-based branch and check Generalized Benders Decomposition algorithm for the efficient solution of such problems. We use machine learning to approximate the effect of the complicating variables on the subproblem by approximating the Benders cuts without solving the subproblem, therefore, alleviating the need to solve the subproblem multiple times. The proposed approach is applied to a mixed integer economic MPC case study on the operation of chemical processes. We show that the proposed algorithm always finds feasible solutions to the optimization problem, given that the mixed integer MPC problem is feasible, and leads to a significant reduction in solution time (up to $97 \%$ or $50 \times$) while incurring small error (in the order of $1\%$) compared to the application of standard and accelerated Generalized Benders Decomposition. 
\end{abstract}

\end{frontmatter}

\section{Introduction}
Model Predictive Contol (MPC) is a common optimization-based strategy for controlling a system such that a desired objective function is optimized subject to constraints that describe the behavior of the system \cite{rawlings2017model}. MPC has usually been applied to continuous dynamical systems, with the objective being either the control performance or some economic metric leading to the so-called economic MPC \cite{ellis2014tutorial}. The implementation of MPC relies on the efficient and robust online solution of the underlying optimization problem. Significant advances have been made in the solution of continuous optimization problems, yet the solution of mixed integer optimization problems \cite{conforti2014integer, boukouvala2016global}, i.e., problems that consider both continuous and discrete (integer) variables, remains a challenge.

Discrete variables can arise either due to the hybrid nature of the dynamic system (e.g., piece-wise affine dynamics) or due to the presence of discrete variables related to the operation of the system \cite{bemporad1999control}. Typical examples of the latter include the operation of energy systems where a unit is either on or off \cite{deng2014model}, chemical processes that can operate only at specific operating points \cite{daoutidis2018integrating}, motion planning in the presence of obstacles \cite{cauligi2021coco}, and fuel cell operation \cite{bertsimas2022online}. These problems are formally known as mixed integer optimal control problems \cite{bemporad1999control, mcallister2022advances} and are formulated as Mixed Integer Dynamic Optimization problems (MIDO). 
A common approach to solving such problems is to discretize the differential equations that describe the system and then use mathematical optimization techniques to solve the resulting Mixed Integer Programming (MIP) problem.

Solving such MIP problems online can be a daunting computational task and thus a major limitation for the implementation of mixed integer MPC. Branch and Bound is a standard solution strategy, where branching is performed either on the integer variables or on the spatial domain of the feasible region for nonconvex problems \cite{bemporad1999efficient, tawarmalani2005polyhedral, bemporad2018numerically, hespanhol2019structure, axehill2006mixed, buchheim2016feasible, naik2017embedded, marcucci2020warm}. This approach can guarantee global optimality, yet it is generally slow for online applications, especially when the underlying problem is a Mixed Integer Nonlinear Programming (MINLP) problem. Multiparametric programming can be used to efficiently determine the optimal solution of a MIP problem \cite{kouramas2011explicit, borrelli2005dynamic, oberdieck2015explicit,axehill2014parametric}. However, computing the critical regions, i.e., the optimal solution, for every value of the parameters of the MIP problem is computationally expensive for nonlinear systems with many states. Another solution approach is to exploit the underlying structure of the problem and use decomposition-based solution algorithms, such as the combinatorial integral approximation (CIA) \cite{sager2012integer}, dual dynamic programming \cite{kumar2021dual}, Alternating Direction Method of Multipliers \cite{takapoui2020simple}, and Generalized Benders Decomposition (GBD) \cite{mohideen1997towards, mitrai2022adaptive, menta2020learning, warrington2019learning}.

GBD is based on the observation that if a subset of the variables of the optimization problem, called complicating variables, are fixed, then the problem can be solved faster. GBD decomposes the MIP problem into a master problem which considers the integer (and possibly some continuous) variables and a subproblem which is a continuous optimization problem. This decomposition is based on the underlying structure of the problem, which in general can be inferred using structure detection methods from network science \cite{mitrai2022stochastic, mitrai2021efficient}. The master problem and subproblem are solved iteratively and coordinated via cuts which inform the former about the effect of the complicating variables on the latter. 

Although decomposition-based methods generally lead to a reduction in solution time compared to monolithic ones, their online implementation is nontrivial since the master problem and the subproblem must be solved multiple times online. This approach can be especially challenging when the subproblem is nonlinear since, in this case, an initial guess for the solution is required. However, identifying a good initial guess is difficult since the parameters of the subproblem, i.e., the complicating variables, can change significantly between two iterations of GBD. Also, the convergence of GBD can be slow and is guaranteed only when the value function of the subproblem is convex \cite{geoffrion1972generalized}. 

Recently, machine learning (ML) has been used to improve the computational performance of optimization algorithms either by focusing on tuning the optimization algorithm or accelerating certain algorithmic steps \cite{bengio2021machine}. In the context of mixed integer MPC, ML has been used to approximate the controller itself \cite{karg2018deep}, the underlying dynamic model \cite{hu2023online}, the values of the discrete variables \cite{masti2020learning, zhu2020fast, cauligi2022prism, russo2023learning}, and the active constraints \cite{cauligi2021coco, bertsimas2022online}. The prediction of the discrete variables transforms the MIP problem into a continuous optimization problem that can, in principle, be solved faster. Alternatively, ML can be used to accelerate the solver itself either by warm-starting or improving algorithmic steps such as branching and cut selection for GBD \cite{mitrai2023focapo}, or can guide the selection of the best solution strategy \cite{kruber2017learning,mitrai2023graph, chakrabarty2021learning}.

In this work, we aim to reduce the computational time related to the solution of the master problem and the subproblem, and thus enable the efficient online application of GBD in mixed integer MPC problems. Specifically, we propose a branch and check algorithm where the master problem is solved only once and Benders cuts are added whenever an integer feasible solution is found during branch and bound. We use ML to approximate the information that is required for the construction of cuts added to the master problem, therefore, alleviating the need for iterative solution of the subproblem. The proposed algorithm can identify a feasible solution to the original problem, if it exists, and can be applied to a broad class of optimization problems that frequently arise in mixed integer MPC applications, such as Mixed Integer Linear, Quadratic, and Nonlinear Programming problems. 

We illustrate the application and advantages of the proposed approach through a case study on the scheduling and dynamic optimization of a chemical process. Specifically, we consider a continuously stirred isothermal reactor that can manufacture multiple products and a mixed integer economic MPC problem is repeatedly solved online to compensate for updated process information, such as a change in product demand and the inlet conditions of the reactor. The results obtained for different numbers of products show that the proposed approach leads to a significant reduction in solution time (up to $97 \%$), while incurring small error. 

The rest of the document is organized as follows: In Section~\ref{bnc GBD}, we present the branch and check GBD algorithm, in Section~\ref{problem statement} we present the mixed integer MPC problem considered, and in Section~\ref{case studies} we compare the proposed approach with standard GBD approaches in numerical case studies.

\section{Machine learning based branch and Benders cut algorithm} \label{bnc GBD}
\subsection{Generalized Benders Decomposition}
We consider the following optimization problem:
\begin{equation}\label{full model}
    \begin{aligned} 
        P(p) := \minimize_{x,y,z} \ \ & f_1 (z,y;p) + f_2 (y,x;p) \\
        \subt \ \ & g_1 (z,y;p) \leq 0 \\
        & g_2 (y,x;p) \leq 0 \\
        & x \in \mathbb{R}^{n_x}, y \in \mathbb{R}^{n_{y}^c} \times \mathbb{Z}^{n_{y}^{d}}, z \in \mathbb{R}^{n_z},
    \end{aligned}
\end{equation}
where $p$ are the parameters of the problem. Depending on the number of continuous variables $n_{x}+n_{z}+n_{y}^c$ and discrete variables $n_{y}^d$, and the number as well as the functional form of the objective and constraints, the solution of the above problem can be computationally challenging. 

GBD is based on the observation that if a subset of the variables is fixed, then the problem is easier to solve. Given the problem in Eq.~\ref{full model}, if the $x$ and $y$ variables are fixed, the following problem called the subproblem is obtained:
\begin{equation}\label{sub problem}
    \begin{aligned} 
        \mathcal{S}(y,p) := \minimize_{\bar{y},x} \ \ & f_2 (\bar{y},x;p) \\
        \subt \ \ & g_2 (\bar{y},x;p) \leq 0 \\
        & \bar{y} = y \ \ : \ \lambda \\
        & \bar{y} \in \mathbb{R}^{n_{y}^c} \times \mathbb{R}^{n_{y}^{d}}, x \in \mathbb{R}^{n_x},
    \end{aligned}
\end{equation}
where $\lambda$ are the Lagrange multipliers for the equality constraints $\bar{y}=y$. From this formulation, it follows that the solution of the subproblem and the optimal value of the dual variables $\lambda$ depend on the values of the complicating variables $y$ and parameters $p$. Given this subproblem, the original problem $P(p)$ (Eq.~\ref{full model}) can be written as:
\begin{equation}\label{full model with val fun}
    \begin{aligned} 
        P(p) := \minimize_{y,z} \ \ & f_1 (z,y;p) + \mathcal{S}(y,p) \\
        \subt \ \ & g_1 (z,y;p) \leq 0 \\
        & z \in \mathbb{R}^{n_z}, y \in \mathbb{R}^{n_{y}^c} \times \mathbb{Z}^{n_{y}^{d}}.
    \end{aligned}
\end{equation}
This problem can not be solved directly since the value function of the subproblem $\mathcal{S}$ is not known explicitly. An approach to overcome this is to approximate the value function via inequalities, called Benders cuts \cite{geoffrion1970elements, geoffrion1970elementspart2, geoffrion1972generalized}, which are equal to
\begin{equation}
    \mathcal{S}(y,p) \geq \mathcal{S}(\bar{y}^l,p) - \lambda(\bar{y}^l;p) (y - \bar{y}^l) \ \forall l \in \mathcal{L},
\end{equation}
where $\mathcal{L}$ denotes the set of points, i.e., values of $\bar{y}$, used to approximate the value function. Given this approximation, the problem in Eq.~\ref{full model with val fun} can be written as:
\begin{equation}\label{full model with val fun - master prob}
    \begin{aligned} 
        M(p, \mathcal{L}) := \minimize_{y,z} \ \ & f_1 (z,y;p) + \eta \\
        \subt \ \ & g_1 (z,y;p) \leq 0 \\
        & \eta \geq \mathcal{S}(\bar{y}^l,p) - \lambda(\bar{y}^l;p) (y - \bar{y}^l) \ \forall l \in \mathcal{L}\\
        & z \in \mathbb{R}^{n_z}, y \in \mathbb{R}^{n_{y}^c} \times \mathbb{Z}^{n_{y}^{d}},
    \end{aligned}
\end{equation}
assuming that the subproblem is feasible for all values of the complicating variables. This problem is known as the master problem and its solution depends on the values of the parameters $p$ and the number of Benders cuts used to approximate the value function $\mathcal{S}$. We note that if the subproblem is infeasible for a given value of the complicating variables, then Benders feasibility cuts can be generated and added to the master problem \cite{geoffrion1972generalized}. 

Given the number and type of complicating variables, a large number of Benders cuts can be potentially used to approximate the value function. However, adding a large number of cuts directly is computationally intractable. The GBD algorithm constructs this approximation iteratively. Specifically, in the first iteration, the master problem is solved without any cuts. The solution of this problem provides a lower bound and values for the complicating variables. Next, the subproblem is solved for the given values of the complicating variables and an upper bound as well as a Benders cut is obtained and added to the master problem. This procedure continues until the bounds converge. This iterative procedure can limit the computational performance of the algorithm in cases where the master problem is a complex MILP and the subproblem is a Nonlinear Programming problem (NLP). One approach to overcome these limitations is to follow a branch and check or branch and Benders-cut solution strategy \cite{thorsteinsson2001branch, laporte1993integer}. 

\subsection{Branch and check solution approach}
In this approach, the master problem is solved once using standard branch and bound/cut algorithms. Specifically, branch and check starts by solving the continuous relaxation of the master problem which generates a set of open nodes $R$, and standard node selection techniques are used to select an open node $\rho$ and solve the continuous relaxation at this node. If the solution at node $\rho$ is integral, then the subproblem (Eq.~\ref{sub problem}) is solved for $y=y^{\rho}$ and the Benders cut
\begin{equation} \label{bc benders cut}
    \eta \geq \mathcal{S}(y^{\rho},p) - \lambda(y^{\rho};p) (y - y^{\rho})
\end{equation}
is added to all the open nodes $R$ in the branch and bound tree; the solution continues by adapting the set of open nodes and selecting a new node. 

\subsection{Machine learning based branch and check Generalized
Benders Decomposition}

In the branch and check solution approach, although the master problem is solved only once, multiple subproblems must be solved to generate the cuts. This is computationally expensive for cases where the subproblem is either a large-scale problem or a nonlinear optimization problem. In the latter case, the solution of NLP problems requires a starting point. However, making a good initial guess is challenging since the parameters of the subproblem, i.e., the values of the complicating variables, can change significantly between two iterations of the algorithm. 
The solution of the subproblem provides information that is necessary for the construction of Benders cuts, i.e., the value function of the subproblem $\mathcal{S}$ and the Lagrange multipliers $\lambda$ for a given value of the complicating variables. 
 
To overcome the limitations mentioned above, we will use ML-based surrogate models $\hat{\mathcal{S}}, \hat{\lambda}$ to approximate the value function $\mathcal{S}$ and Lagrange multipliers $\lambda$ of the subproblem. These models can be learned offline by solving the subproblems for multiple values of the complicating variables $\{y_i\}_{i=1}^{N_{data}}$ and obtaining the value function $\{\mathcal{S}_i (y_i) \}_{i=1}^{N_{data}}$ and Lagrange multipliers $\{\lambda_i (y_i) \}_{i=1}^{N_{data}}$. Through this procedure, we obtain two datasets: the first one $\{y_i, \mathcal{S}_{i} (y_i)\}_{i=1}^{N_{data}}$ is used to learn a surrogate $\hat{\mathcal{S}}(y)$ for the value function and the second one $\{y_i, \lambda_i (y_i) \} _{i=1}^{N_{data}}$ is used to learn a surrogate $\hat{\lambda}(y)$ for the Lagrange multipliers. Given these surrogate models, once an integer feasible solution is found during branch and check at node $\rho$ the following cut is added
\begin{equation}
    \eta \geq \hat{\mathcal{S}}(y^\rho) - \hat{\lambda}(y^\rho) (y - y^\rho)
\end{equation}
to all the open nodes. The proposed algorithm is presented in Algorithm~\ref{alg: ml based bnc benders algorithm}. 

\begin{algorithm}[t]
\caption{ML-based Branch and Check GBD}
\label{alg: ml based bnc benders algorithm}
\KwData{Optimization problem, surrogate models $\hat{\mathcal{S}}$ and $\hat{\lambda}$}
\KwResult{Solution of the optimization problem}
Start Branch and Bound algorithm by solving continuous relaxation of the master problem\;
Obtain open nodes $R$\;
\While{$R \neq \emptyset$}{
Select a node $\rho$ from $R$ using node selection techniques\;
Solve continuous relaxation at $\rho$ and obtain $y^{\rho}$\;
\eIf{$y^{\rho}$ is integer}{
Approximate the value function $\hat{\mathcal{S}}(y^{\rho})$\;
Approximate Lagrange multipliers $\hat{\lambda}(y^{\rho})$\;
Add Benders cut to all open nodes in $R$\\ 
\begin{equation*}
    \eta \geq \hat{S}(y^{\rho}) - \hat{\lambda}(y^{\rho}) (y - y^{\rho})
\end{equation*}
Add MIP-based cuts\;
Update existing nodes in $R$}{
Partition domain of $y$ variables\;
Add MIP-based cuts\;
Update existing nodes in $R$}
}
\end{algorithm} 

\begin{remark}
    \normalfont The proposed algorithm can be used for the solution of a wide class of optimization problems that arise in mixed integer MPC applications, such as MILPs ($f,g$ are affine), MIQPs ($f$ is quadratic positive definite and $g$ affine) as well as MINLPs where either the continuous relaxation is convex or the value function of the subproblem is convex.
\end{remark}

\section{Mixed Integer Model Predictive Control for Integrated Scheduling and Optimization of a Chemical Process} \label{problem statement}

We will consider a continuous chemical production system such that a number ($N_p$) of different products can be manufactured over a production horizon $H$. We consider the case where the system is originally following a nominal schedule, and at some time point $T_0$ a disturbance affects the system. The disturbance can either be a change in the demand of the products or a change in the conditions of the manufacturing system. Both types of disturbances affect the available production time, since a change in demand affects the production time of the products and a change in the inlet conditions of the system increases the transition times which subsequently reduces the production time given the fixed time horizon. Once the disturbance affects the system, a mixed integer MPC problem is solved to determine the production sequence, production time, and transition times as presented in Fig.~\ref{fig: rescheduling shceme}.
\begin{figure*}[h!]
    \centering
    \includegraphics[scale=0.6]{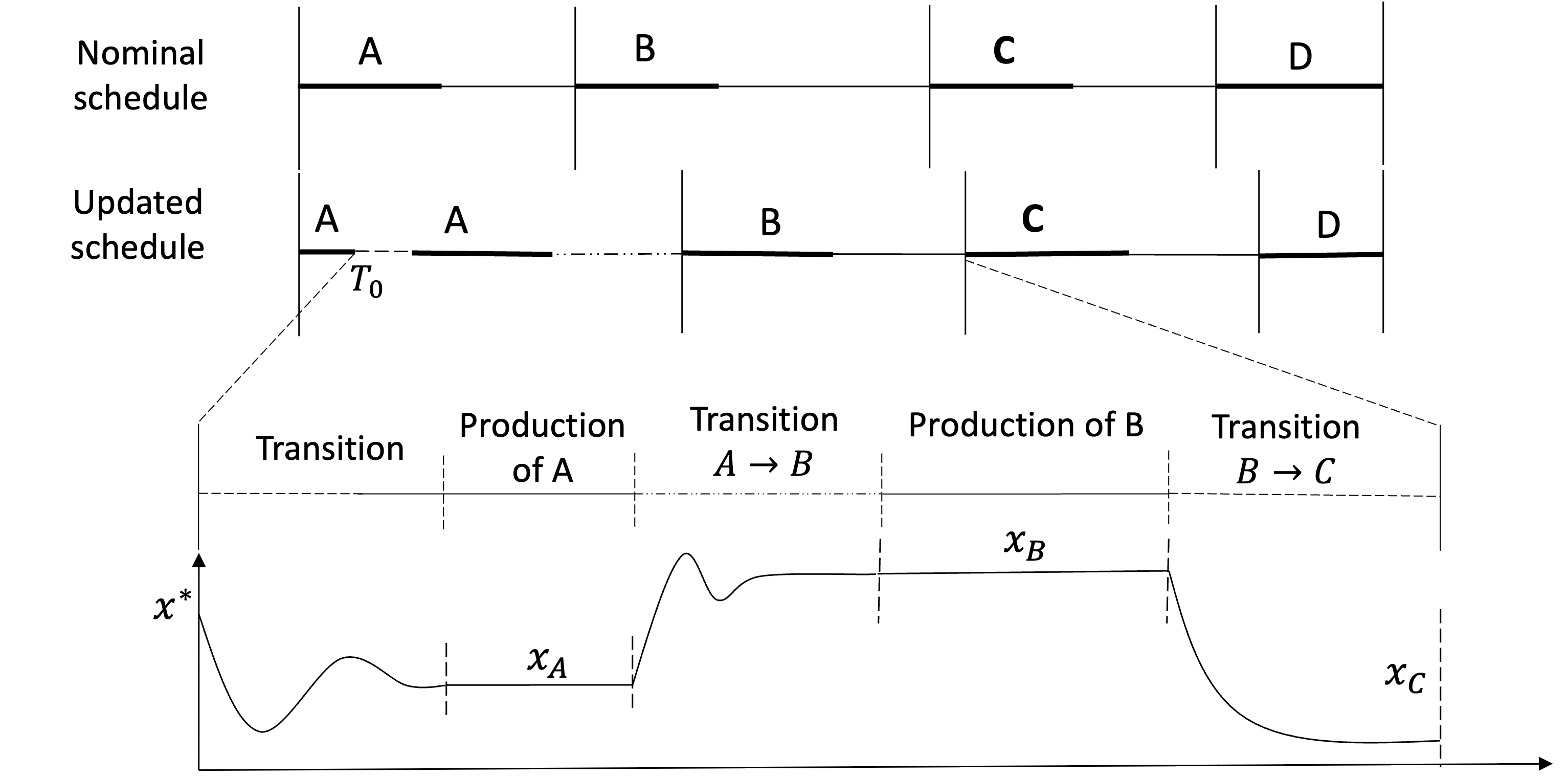}
    \caption{Schematic of mixed integer MPC  problem}
    \label{fig: rescheduling shceme}
\end{figure*}

\subsection{Scheduling model}

We define $\mathcal{I}_p = \{1,.., N_p\}$ to be the set of products that must be manufactured and the time horizon is discretized into $N_s$ slots ($\mathcal{I}_s = \{1,.., N_s\}$). Every slot, except the first one, has two regimes (see Fig.~\ref{fig: rescheduling shceme}); a production regime where a product is manufactured and a transition regime where a transition occurs between the operating points of two products. The first slot has three regimes, a transition regime from some intermediate state to the operating point of the product manufactured in the first slot, a production regime, and a second transition regime from the operating point of the product manufactured in the first slot to the operating point of the product manufactured in the second slot. 

We define a binary variable $W_{ik}$ which is equal to one if product $i$ is manufactured in slot $k$ and zero otherwise, and a binary variable $Z_{ijk}$ which is equal to one if a transition occurs from product $i$ to $j$ in slot $k$ and zero otherwise. We also define a binary variable $\hat{Z}_i$ which is equal to one if a transition occurs from an intermediate state to the operating point of the product manufactured in the first slot. We assume that every product is manufactured only once, i.e., $N_p=N_s$, and that in every slot only one product is manufactured. These constraints are described mathematically as follows:
\begin{equation} \label{logic cons}
    \begin{aligned}
        \sum_{k \in \mathcal{I}_s} W_{ik} & = 1 \ \  \forall \ i \in \mathcal{I}_p \\
        \sum_{i \in \mathcal{I}_p} W_{ik}  & = 1 \ \  \forall \ k \in \mathcal{I}_s\\
        Z_{ijk} & \geq W_{ik} + W_{j,k+1} -1 \ \forall i \in \mathcal{I}_p,j \in \mathcal{I}_p ,k \in \mathcal{I}_s \setminus \{N_s\}\\
        \hat{Z}_i & = W_{i1} \ \ \forall i \in \mathcal{I}_p.
    \end{aligned}
\end{equation}
 $T_k^s, T_k^e$ are the starting and ending time of slot k,  $\Theta_{ik}$ is the production time of product $i$ in slot $k$, $\theta_{ijk}$ is the transition time from product $i$ to $j$ in slot $k$, and $\hat{\theta}_i$ is the transition time from an intermediate state to the steady state of product $i$. The timing constraints are:
\begin{equation}\label{timing cons}
    \begin{aligned}
        T_1^s & = 0\\
        T_k^e & = T_{k}^s + \sum_{i \in \mathcal{I}_p} \Theta_{ik} + \theta_k^t \ \ \forall k \in \mathcal{I}_s\\
        T_{N_s}^e & = H-T_0\\
        T^{s}_{k+1} & = T^{e}_k \ \ \forall k \in \mathcal{I}_s \setminus \{N_s\} \\
        \Theta_{ik} & \leq W_{ik} H \ \ \forall i \in \mathcal{I}_p, k \in \mathcal{I}_s
    \end{aligned}
\end{equation}
and the transition times are equal to
\begin{equation}\label{timing cons 2}
    \begin{aligned}
        \theta_k^t & = \sum_{i \in \mathcal{I}_p} \sum_{j \in \mathcal{I}_p} \theta_{ijk} Z_{ijk} \ \ \forall k \in \mathcal{I}_s \setminus \{1\} \\
        \theta_1^t & = \sum_{i \in \mathcal{I}_p} \sum_{j \in \mathcal{I}_p} \theta_{ij1} Z_{ij1} + \sum_{i \in \mathcal{I}_p} \hat{\theta}_i \hat{Z}_i
    \end{aligned}
\end{equation}
\begin{equation}\label{timing cons - tr time}
    \begin{aligned}
        \theta_{ijk} & \geq \theta^{min}_{ij} \ \ \forall i \in \mathcal{I}_p, j \in \mathcal{I}_p, k \in \mathcal{I}_s\\
        \hat{\theta}_{i} & \geq \hat{\theta}^{min}_{i} \ \ \forall i \in \mathcal{I}_p,
    \end{aligned}
\end{equation}
where $\theta_{ij}^{min}$ is the minimum transition time from product $i$ to $j$, and $\hat{\theta}_{i}^{min}$ is the minimum transition time from the intermediate state to the steady state of product $i$. The production rate of product $i$ is $r_i$, the production amount of product $i$ in slot $k$ is $q_{ik}$, and the inventory of product $i$ in slot $k$ is $I_{ik}$. The production constraints are
\begin{equation}\label{inv cons}
    \begin{aligned}
        q_{ik} & = r_i \Theta_{ik} \ \ \forall i \in \mathcal{I}_p, k\in \mathcal{I}_s \\
        I_{i,1} & = I^0_i + q_{i1} - S_{i,1} \ \ \forall i \in \mathcal{I}_p\\
        I_{i,k+1} & = I_{ik} + q_{ik} - S_{ik}\ \ \forall i \in \mathcal{I}_p, k \in \mathcal{I}_s \setminus \{N_s\},
    \end{aligned}
\end{equation}
where $I^0_i$ is the initial inventory of product $i$ and $S_{ik}$ is the amount of product $i$ sold in slot $k$. We assume that the due date for the demand is the end of the time horizon, i.e., $S_{iN_S} \geq d_i \ \ \forall i \in \mathcal{I}_p$. 

\subsection{Dynamic model}
The dynamic behavior of the system is described by a system of ordinary differential equations
\begin{equation}
    \frac{dx(t)}{dt} = F(x(t),u(t)),
\end{equation}
where $x \in \mathbb{R}^{n_x}$ are the states of the dynamic model, $u \in \mathbb{R}^{n_u}$ are the manipulated variables, and $F: \mathbb{R}^{n_z} \times \mathbb{R}^{n_u} \mapsto \mathbb{R}^{n_x}$ are vector functions. The system can manufacture multiple products, which correspond to different steady state operating points $\{x^{ss}_i,u^{ss}_i \}_{i=1}^{N_p}$, by adjusting the manipulated variables. 

We consider all the transitions between the products simultaneously and discretize the differential equations using orthogonal collocation on finite elements using $N_{fe}$ $(\mathcal{I}_{f} = \{1,..., N_{fe}\})$ finite elements and $N_c$ $(\mathcal{I}_c = \{1,..., N_c\})$ collocation points. We define $x_{ijkfc}$ and $u_{ijkfc}$ as the values of the state and manipulated variable for a transition from product $i$ to $j$ in slot $k$ at finite element $f$ and collocation point $c$. We also define variables $\hat{x}_{ifc}$, $\hat{u}_{ifc}$ as the values of the state and manipulated variable for a transition from state $x_0$ to product $i$ to at finite element $f$ and collocation point $c$ at the first slot. The discretized differential equations for the transitions between the products are
\begin{equation} \label{disc ode cons}
 \begin{aligned}
 &  x_{ijfck} = x0_{ijfk} + h^{fe}_{ijk} \sum_{m=1}^{N_{cp}} \Omega_{mc} \dot{x}_{ijfmk}^n \ \ \ \forall (i,j,k,f,c) \in \mathcal{I}_{dyn}^{1}\\
 &  h^{fe}_{ijk} = \frac{\theta_{ijk}}{N_{fe}} \ \ \forall i \in \mathcal{I}_p, j \in \mathcal{I}_p, k\in \mathcal{I}_s \\
 &  x0_{ijfk}= x0_{ij,f-1,k} + h^{fe}_{ijk} \sum_{m=1}^{N_{cp}} \Omega_{mc} \dot{x}_{ij,f-1,mk} \ \ \forall (i,j,k,f,c) \in \mathcal{I}_{dyn}^{2} \\
 & \dot{x}_{ijfck} = F(x_{ijfck},u_{ijfckz}) \ \ \forall (i,j,k,f,c) \in \mathcal{I}_{dyn}\\
 & t^d_{ijfck} = h_{ijk}^{fe} (f-1+\gamma_c)\ \ \forall (i,j,k,f,c) \in \mathcal{I}_{dyn}\\
  & x0_{ij1k} = x^{ss}_{i} \ \ \forall i \in \mathcal{I}_p, j \in \mathcal{I}_p, k\in \mathcal{I}_s\\
  & x_{ijkN_{fe}N_{cp}} = x^{ss}_{j} \ \ \forall i \in \mathcal{I}_p, j \in \mathcal{I}_p, k\in \mathcal{I}_s\\
  & u_{ijk11} = u^{ss}_{i} \ \ \forall i \in \mathcal{I}_p, j \in \mathcal{I}_p, k\in \mathcal{I}_s\\
  & u_{ijkN_{fe}N_{cp}} = u^{ss}_{j} \ \ \forall i \in \mathcal{I}_p, j \in \mathcal{I}_p, k\in \mathcal{I}_s,
 \end{aligned}
 \end{equation}
where $\mathcal{I}_{dyn}^{1} = \{ (i,j,k,f,c) \ \forall i \in \mathcal{I}_p, j \in \mathcal{I}_p, f\in \mathcal{I}_f, c\in \mathcal{I}_c, k\in \mathcal{I}_s \}$, $\mathcal{I}_{dyn}^{2} = \{ (i,j,k,f,c) \ \forall i \in \mathcal{I}_p, j \in \mathcal{I}_p, f \in \mathcal{I}_f \setminus \{1,2\}, c\in \mathcal{I}_c, k\in \mathcal{I}_s \}$, $t_{ijfck}$ is the time at finite element $f$ at collocation point $c$ for a transition from product $i$ to $j$ in slot $k$, $\Omega$ is the collocation matrix, and $\gamma$ are the Radau roots. Similarly, the discretized equations for the transition from the intermediate state $x_0$ to the steady state operating point of the products are
\begin{equation} \label{disc ode cons int trans}
 \begin{aligned}
 &  \hat{x}_{ifc} = \hat{x}0_{if} + \hat{h}^{fe}_{i} \sum_{m=1}^{N_{cp}} \Omega_{mc} \dot{\hat{x}}_{ifm} \ \ \ \forall (i,f,c) \in \hat{\mathcal{I}}_{dyn}^{1}\\
 &  \hat{h}^{fe}_{i} = \frac{\hat{\theta}_{i}}{N_{fe}} \ \ \forall i \in \mathcal{I}_p\\
 &  \hat{x}0_{if} = \hat{x}0_{i,f-1} + \hat{h}^{fe}_{i} \sum_{m=1}^{N_{cp}} \Omega_{mc} \dot{\hat{x}}_{i,f-1,m} \ \ \forall (i,f,c) \in \hat{\mathcal{I}}_{dyn}^{2} \\
 & \dot{\hat{x}}_{ifc} = F(\hat{x}_{ifc},\hat{u}_{ifc}) \ \ \forall (i,f,c) \in \hat{\mathcal{I}}_{dyn}^{1}\\
 & \hat{t}^d_{ifc} = \hat{h}_{i}^{fe} (f-1+\gamma_c)\ \ \forall (i,f,c) \in \hat{\mathcal{I}}_{dyn}^{1}\\
  & \hat{x}0_{i1} = x_0 \ \ \forall i \in \mathcal{I}_p\\
  & \hat{x}_{iN_{fe}N_{cp}} = x^{ss}_{i} \ \ \forall i \in \mathcal{I}_p\\
  & u_{iN_{fe}N_{cp}} = u^{ss}_{i} \ \ \forall i \in \mathcal{I}_p, 
 \end{aligned}
 \end{equation}
where $\hat{\mathcal{I}}_{dyn}^{1} = \{ (i,f,c) \ \forall i \in \mathcal{I}_p, f\in \mathcal{I}_f, c\in \mathcal{I}_c \}$, $\hat{\mathcal{I}}_{dyn}^{2} = \{ (i,f,c) \ \forall i \in \mathcal{I}_p, f \in \mathcal{I}_f \setminus \{1,2\}, c\in \mathcal{I}_c, \}$, $\hat{t}_{ifc}$ is the time at finite element $f$ at collocation point $c$ for a transition from $x_0$ to the operating point of product $i$.

\subsection{Integrated problem}
The integrated problem seeks to identify the optimal production sequence, production times, transition times, and dynamic transitions between products such that the profit is maximized while the demand and other operational constraints are satisfied. The objective function is the profit and has two terms: the first is the sales minus the operating and inventory cost and the second is the cost related to the dynamic transitions between the products. The three terms are:
\begin{equation}
    \begin{aligned}
        \Phi_1 &= \sum_{ i \in \mathcal{I}_p } \sum_{k \in \mathcal{I}_s} \bigg( P_{ik} S_{ik}- C_{ik}^{op} q_{ik} \bigg)  - C^{inv} I_{ik} - \sum_{i\in \mathcal{I}_p}  \sum_{j\in \mathcal{I}_p}   \sum_{k\in \mathcal{I}_s}  C_{ij}^{tr} Z_{ijk}\\
       \Phi_2 & = \sum_{i\in \mathcal{I}_p}  \sum_{j\in \mathcal{I}_p}   \sum_{k\in \mathcal{I}_s} Z_{ijk} f_{ijk}^{dyn}\ - \sum_{i\in \mathcal{I}_p} \hat{Z}_{i} \hat{f}_{i}^{dyn}\\
        f_{ijk}^{dyn} & = \alpha_u \bigg( \sum_{l \in \mathcal{I}_f} \sum_{c \in \mathcal{I}_c} N_{fe}^{-1}t_{ijfck} \Omega_{c, N_c} (u_{ijfck}-u_{j}^{ss})^2\bigg)\\
        \hat{f}_{i}^{dyn} & = \alpha_u \bigg( \sum_{l \in \mathcal{I}_f} \sum_{c \in \mathcal{I}_c} N_{fe}^{-1} \hat{t}_{ijfck} \Omega_{c, N_c} (\hat{u}_{ijfck}-u_{j}^{ss})^2\bigg)
    \end{aligned}
\end{equation}
where $P_{ik}$,  $C_{ik}^{op}$ are the price and operating cost of product $i$ in slot $k$ respectively, $C^{inv}$ is the inventory cost, $C_{ij}^{tr}$ is the fixed transition cost from product $i$ to $j$, and $\alpha_u$ is a weight coefficient. Overall, the resulting problem is a mixed integer economic MPC problem and the optimization problem is
\begin{equation} \label{intgr prob}
    \begin{aligned}
      P(p):= \maximize \ \ & \Phi_1 - \Phi_2  \\
        \subt \ \ & \text{Eq.}~\ref{logic cons}, \ref{timing cons}, \ref{timing cons 2}, \ref{timing cons - tr time}, \ref{inv cons}, \ref{disc ode cons}, \ref{disc ode cons int trans},
    \end{aligned}
\end{equation}
where $p$ are the parameters of the optimization problem. 

\subsection{Decomposition-based solution approach}
Given the mixed integer MPC problem, we observe that if the scheduling variables and transition times are fixed, then the problem is decomposed into a number of independent subproblems, where one subproblem considers the transition from the intermediate state to the steady state of the product manufactures in the first slot and the other subproblem considers transitions between products. The subproblem for a transition between two products is 
\begin{equation}
    \begin{aligned}
    \mathcal{S}_{ijk} (\theta_{ijk}): =\minimize_{x_{ijkfc},u_{ijkfc}, \tilde{\theta}_{ijk}}  \ \ & \ \ f_{dyn}^{ijk} (x_{ijkfc},u_{ijkfc}, \tilde{\theta}_{ijk})  \\
    \subt \ \ & \ \ g_{dyn}(\tilde{\theta}_{ijk}, x_{ijkfc},u_{ijkfc}) \leq 0 \\
    & \ \ \tilde{\theta}_{ijk} = \theta_{ijk} \ \ : \lambda_{ijk},
    \end{aligned}
\end{equation}
where $\mathcal{S}_{ijk} (\theta_{ijk})$ is the value function and $\lambda_{ijk}$ is the Lagrange multiplier of the equality constraint $\tilde{\theta}_{ijk} = \theta_{ijk}$. The value of the Lagrange multiplier at the optimal solution of the subproblem depends on the transition time, i.e., $\lambda_{ijk}(\theta_{ijk})$. The subproblem for the transition from the intermediate state to the operating point of the product $i$ is
\begin{equation}
    \begin{aligned}
    \hat{\mathcal{S}}_{i} (\hat{\theta}_{i}, x_0): =\minimize_{\hat{x}_{ifc},\hat{u}_{ifc},\check{\theta}_{i}} \ \ & \ \ f_{dyn}^{ijk} (\hat{x}_{ifc},\hat{u}_{ifc},\check{\theta}_{i})  \\
    \subt \ \ & \ \ g_{dyn}(\check{\theta}_{i}, \hat{x}_{ifc},\hat{u}_{ifc}) \leq 0 \\
    & \ \ \check{\theta}_{i} = \hat{\theta}_{i} \ \ : \hat{\lambda}_{i}(\hat{\theta}_{i},x_0)\\
    & \ \ x_{i11} = x_0.
    \end{aligned}
\end{equation}
The solution of this subproblem depends on the transition time and the initial condition of the system, thus  $\hat{\mathcal{S}}_i$ and $\hat{\lambda}$ depend on  $\hat{\theta}_i$ and $x_0$.

Given this structure of the problem and the value functions defined above, the mixed integer MPC problem, defined as $P(p,x_0)$, can be written as:
\begin{equation}
    \begin{aligned}
    \maximize_{w, \theta_{ijk}, \hat{\theta}_i} \ & \ \Phi_1 (w;p) - \sum_{ijk} Z_{ijk} \mathcal{S}_{ijk}(\theta_{ijk}) - \sum_{i} \hat{Z}_i \hat{S}_i(\hat{\theta},x_0)\\
    \subt \ & \ g_{sched}(w, \theta_{ijk}, \hat{\theta}_i; p) \leq 0 \ (\text{Eq.}~\ref{logic cons}, \ref{timing cons}, \ref{timing cons 2}, \ref{timing cons - tr time}, \ref{inv cons})
    \end{aligned}
\end{equation}
where $w$ are scheduling variables. This problem can not be solved directly since the value functions are not known explicitly. The value functions can be approximated iteratively with hyperplanes, called Benders cuts, as follows \cite{geoffrion1970elements, geoffrion1971duality, geoffrion1972generalized}:
\begin{equation}
    \begin{aligned}
    \mathcal{S}_{ijk}(\theta_{ijk}) \geq & \mathcal{S}_{ijk}(\bar{\theta}_{ijk}^l) - \lambda_{ijk}^l(\bar{\theta}_{ijk}^l) (\theta_{ijk} - \bar{\theta}_{ijk}^l) \ \ \forall 
    l \in \mathcal{L}  \\
    \hat{\mathcal{S}}_i(\hat{\theta}_i) \geq & \hat{\mathcal{S}}_i (\bar{\hat{\theta}}_i^l,x_0) - \hat{\lambda}_i^l(\bar{\hat{\theta}}_i^l,x_0) (\hat{\theta}_i - \bar{\hat{\theta}}_i^l) \ \ \forall 
    l \in \mathcal{L},
    \end{aligned}
\end{equation}
where $l \in \mathcal{L}$ denotes the number of points used to approximate the value functions. The original problem can now be reformulated as 
\begin{equation}
    \begin{aligned}
    \maximize \ \ & \ \ \Phi_1 (w;p) - \sum_{ijk} Z_{ijk} \eta_{ijk} - \sum_{i} \hat{Z}_i \hat{\eta}_{i} \\
    \subt \ \ & \ \ g_{sched}(w, \theta_{ijk}, \hat{\theta}_i;p) \leq 0 \\
    & \ \     \eta_{ijk} \geq \mathcal{S}_{ijk}^l(\bar{\theta}_{ijk}^l) - \lambda_{ijk}^l (\theta_{ijk} - \bar{\theta}_{ijk}^l) \ \forall i,j,k,l  \\
    & \ \ \hat{\eta}_i \geq  \hat{\mathcal{S}}_i (\bar{\hat{\theta}}_i^l,x_0) - \hat{\lambda}_i^l(\bar{\hat{\theta}}_i^l,x_0) (\hat{\theta}_i - \bar{\hat{\theta}}_i^l) \ \forall i,l.
    \end{aligned}
\end{equation}

Given this problem reformulation, we can apply the proposed ML branch and check GBD algorithm, where the branching procedure explores different production sequences. Once a feasible production sequence is found, the transition cost and Lagrange multipliers are approximated via the Benders cut which are added to the master problem.

\begin{algorithm*}[ht]
\caption{Learning surrogate models for approximating value function and Lagrangean multipliers for transition from intermediate state to the steady state of the different products}
\label{alg: ml based bnc benders algorithm sampling mpc}
\KwData{Set of products $\mathcal{P}$, Steady-state concentration of products $x_{ss}$, Number of points $N_d$}
\KwResult{Surrogate models}
\For{$p \in \mathcal{P}$}{
epoch=0\;
$\mathcal{D}^{c}_{l} = \{ \ \}$, $\mathcal{D}^{c}_{u} = \{ \ \}$, $\mathcal{D}^{\lambda}_{l} = \{ \ \}$, $\mathcal{D}^{\lambda}_{u} = \{ \ \}$\;
\While{$epoch \leq N$}{
Get random value for $x_0 \sim U(0,1)$\;
Get $\hat{\theta}_{min}$ from $x_0$ to $x_{ss}(p)$\;
Select randomly $\hat{\theta} \sim U(\hat{\theta}_{min}, 2\hat{\theta}_{min}$)\;
Solve dynamic optimization problem from $x_0$ to $x_{ss}(p)$ with transition time $\hat{\theta}$ and obtain cost $c$ and Lagrangean multipliers $\lambda$\;
\eIf{$x_0 \leq x_{ss}(p)$}{
Append data point with features $x_0, \hat{\theta}$ and label $c$ to dataset $\mathcal{D}_l^{c}$: $\mathcal{D}_l^{c} = \mathcal{D}_l^{c} \cup ((\hat{\theta},x_0),c)$\;
Append data point with features $x_0, \hat{\theta}$ and label $\lambda$ to dataset $\mathcal{D}_l^{\lambda}$: $\mathcal{D}_l^{\lambda} = \mathcal{D}_l^{\lambda} \cup ((\hat{\theta},x_0),\lambda)$\;
}{Append data point with features $x_0, \hat{\theta}$ and label $c$ to dataset $\mathcal{D}_l^{c}$: $\mathcal{D}_u^{c} = \mathcal{D}_u^{c} \cup ((\hat{\theta},x_0),c)$\;
Append data point with features $x_0, \hat{\theta}$ and label $\lambda$ to dataset $\mathcal{D}_u^{\lambda}$: $\mathcal{D}_u^{\lambda} = \mathcal{D}_l^{\lambda} \cup ((\hat{\theta},x_0),\lambda)$\;}
epoch = epoch +1\;
}
Learn surrogate model $\tilde{\hat{\mathcal{S}}}_{p}^{u}$ using data set $\mathcal{D}_{u}^{c}$, model $\tilde{\hat{\mathcal{S}}}_{p}^{l}$ using data set $\mathcal{D}_{l}^{c}$, model $\tilde{\hat{\lambda}}_{p}^{u}$ using data set $\mathcal{D}_{u}^{\lambda}$, model $\tilde{\hat{\lambda}}_{p}^{l}$ using data set $\mathcal{D}_{l}^{\lambda}$\;
}
\end{algorithm*} 

\subsection{Feasibility of the mixed integer MPC problem} \label{feas and opt of mip mpc}

The following theorem provides conditions for determining the feasibility of the mixed integer MPC problem for this class of problems.
\begin{theorem}
    The mixed integer optimization problem in Eq.~\ref{intgr prob} is feasible if the following condition holds:
    \begin{equation} \label{feasibility condition}
        \sum_{i \in \mathcal{I}_p} d_i/r_i + \sum_{ (i,j) \in \mathcal{I}_{trans}} \theta_{ij}^{min} + \hat{\theta}_{i^*} \leq H-T_0,
    \end{equation}
    where $\mathcal{I}_{trans}$ contains the transitions that occur between products and $i^{*}$ is the product that is manufactured in the first slot such that the total transition time is minimized. 
\end{theorem}
\begin{proof}
    The right-hand side in the inequality of Eq.~\ref{feasibility condition} is the total available time and the left-hand side is the total time necessary for production and transitions. The first term of the left-hand side is the minimum production time for all the products such that the demand is satisfied. The second term captures the transitions with the minimum transition time. These transitions can be obtained by the solution of the following integer optimization problem:
    \begin{equation}
        \begin{aligned}
            \minimize_{W_{ik},Z_{ijk}, \hat{Z}_i} \ \ & \sum_{i \in \mathcal{I}_p} \hat{Z}_{i} \hat{\theta}_i^{min} + \sum_{i \in \mathcal{I}_p} \sum_{j \in \mathcal{I}_p} \sum_{k \in \mathcal{I}_s} Z_{ijk} \theta_{ij}^{min} \\
            \subt \ \ & \hat{Z}_i = W_{i1} \ \ \forall i \in \mathcal{I}_p\\
            & \sum_{k \in \mathcal{I}_s} W_{ik}  = 1 \ \  \forall \ i \in \mathcal{I}_p \\
        & \sum_{i \in \mathcal{I}_p} W_{ik}   = 1 \ \  \forall \ k \in \mathcal{I}_s\\
        & Z_{ijk}  \geq W_{ik} + W_{j,k+1} -1 \ \forall i \in \mathcal{I}_p,j \in \mathcal{I}_p ,k \in \mathcal{I}_s \setminus N_s,
        \end{aligned}
    \end{equation}
    and the set $I^{trans}$ and $i^*$ are equal to $I^{*} = \{ (i,j)| \{Z_{ijk}=1\}_{k=1}^{N_s}\}$ and $i^{*} = \{ i| \hat{Z}_{i}=1\}$. Therefore, if the summation of the production and transition time is less than length of the time horizon, then the optimization problem is feasible and the system can satisfy the demand.
\end{proof}

Based on Eq.~\ref{feasibility condition}, it follows that the optimization problem considered does not always have a feasible solution, i.e., the system can not always reject the disturbances and satisfy the demand. We posit however that the solution obtained by the proposed solution approach is feasible if the original optimization problem is feasible. Specifically, in the proposed approach, it is assumed that the subproblem is always feasible. Therefore, infeasibility can arise only from the master problem. However, if the master problem is infeasible, then the original problem will also be infeasible. Thus, the solution obtained by the proposed approach is feasible if the master problem is feasible. 

For the chemical production problem considered, since the subproblem considers transitions between products, the only possible source of infeasibility is the transition time, specifically, the case where the transition time is less than the minimum transition time between two products. However, since the transition times are bounded from below $\theta_{ijk} \geq \theta^{min}_{ij}$ and $\hat{\theta}_i \geq \hat{\theta}_{i}^{min}$ (Eq.~\ref{timing cons - tr time}) in the master problem, the subproblems are indeed always feasible. 

\begin{figure}[t]
    \centering
    \includegraphics[trim= 00 275 600 0,clip,scale=0.6]{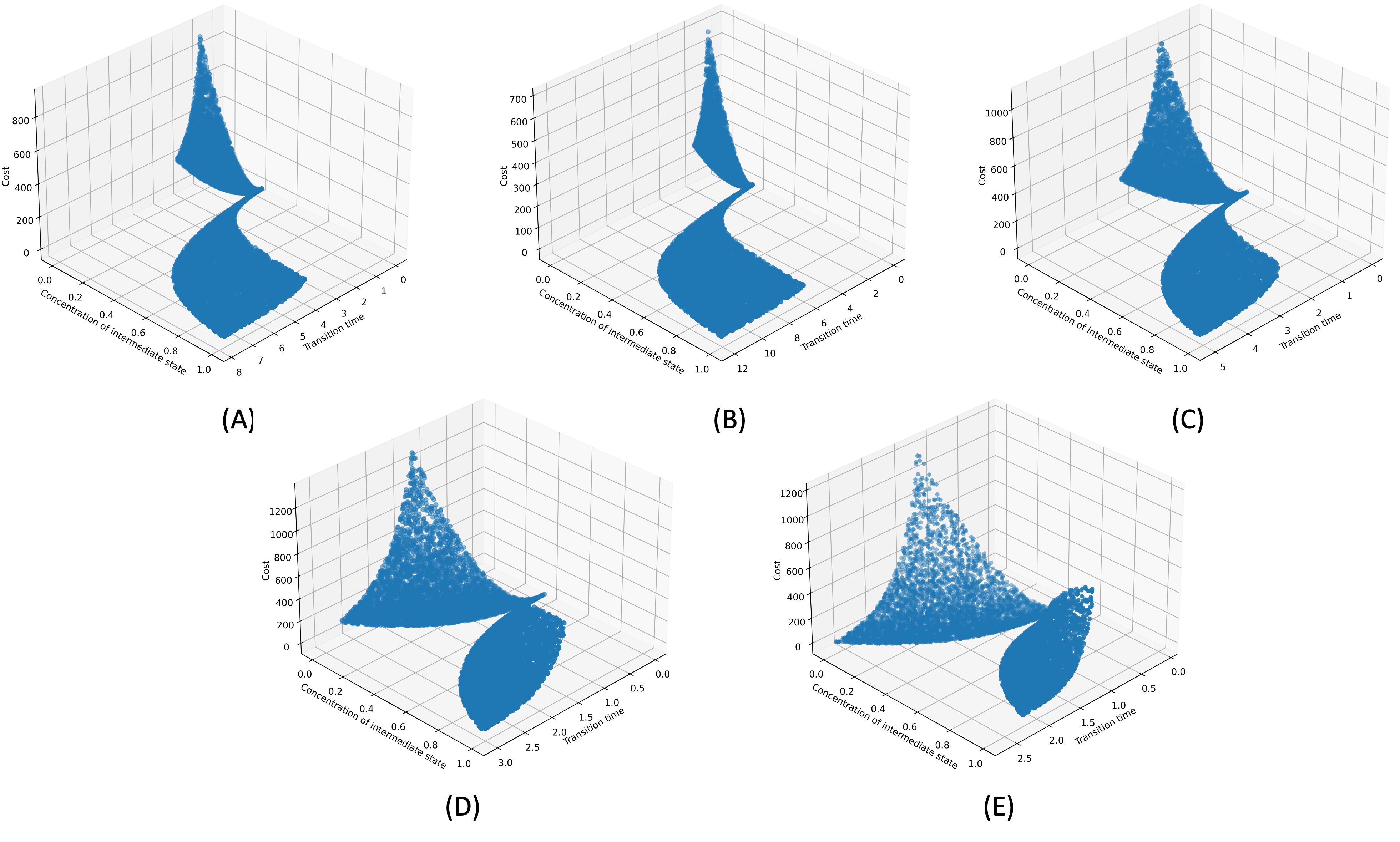}
    \caption{Transition cost from the intermediate state to the steady state of product 1 for different transition times. 
    The x-axis is the transition time, the y-axis the initial concentration in the reactor $(x_0)$, and the z-axis is the transition cost divided by 1000.}
    \label{fig: MIP MIC value functions}
\end{figure}

\begin{table}[t]
\centering
\caption{Parameters of the optimization problem for the case of five products, $C^{inv}=0.0026 \$/mol, \alpha_u=1$. $F_{ss}, c_{ss}$ are the steady state values of the inlet flowrate and concentration.}
\resizebox{\columnwidth}{!}{
\begin{tabular}{ccccccc}
\hline
Product & Demand $(mol)$ (change) & Price $(\$/mol)$ & Oper. cost & Prod. rate & $F_{ss}$  & $c_{ss}$  \\ \hline
1       & 600 (100)       & 200   & 13         & 150        & 200  & 0.24 \\
2       & 550(15)         & 160   & 22         & 80         & 100  & 0.20 \\
3       & 600 (30)        & 130   & 35         & 278        & 400  & 0.30 \\
4       & 1200 (20)       & 110   & 29         & 607        & 1000 & 0.39 \\
5       & 2000 (400)      & 140   & 25         & 1250       & 2500 & 0.50 \\ \hline
\end{tabular}}
\label{table: optimization problem parameters}
\end{table}

\begin{algorithm*}[t]
\caption{ML-based Branch and Check GBD for mixed-integer MPC for the CSTR}
\label{alg: ml based bnc benders algorithm MIP MPC case}
\KwData{Master problem, surrogate models $\tilde{\hat{\mathcal{S}}}_{p}^{u}$, $\tilde{\hat{\mathcal{S}}}_{p}^{l}$, $\tilde{\hat{\lambda}}_{p}^{u}$, $\tilde{\hat{\lambda}}_{p}^{l}$, $x_0$, $x_{ss}$}
\KwResult{Solution of the optimization problem}
Start Branch and Bound and obtain open nodes $R$\;
\While{$R \neq \emptyset$}{
Select a node $\rho$ from $R$ using node selection techniques\;
Solve continuous relaxation at $\rho$ and obtain $w^{\rho}$, $\theta_{ijk}^{\rho}$, $\hat{\theta}_{i}^{\rho}$\;
\eIf{Binary variables are integer}{
Approximate the value function and Lagrange multipliers for transitions between products that occur $ \mathcal{C}_{c} = \{ (i,j,k) |Z_{ijk}=1\}$, $\tilde{{\mathcal{S}}}(\theta_{ijk})$, $\tilde{{\lambda}}(\theta_{ijk})$\;
Add Benders cuts, $\eta \geq \tilde{{S}}(\theta_{ijk}^{\rho}) - \tilde{{\lambda}}(\theta_{ijk}^{\rho}) (\theta_{ijk} - \theta_{ijk}^{\rho}) \ \ \forall (i,j,k) \in \mathcal{C}_c$,
to all open nodes in $R$\;
Determine the product $p$ that is manufactured in the first slot\;
Obtain transition time $\hat{\theta}_p^{\rho}$\;
\eIf{$x_0 \leq x_{ss}(p)$}{
Approximate the value function and Lagrange multipliers using
$\tilde{\hat{S}}_{p}^{l}, \tilde{\hat{\lambda}}_{p}^{l}$ for the transition from $x_0$ to $p$ and add Benders cut $\hat{\eta}_p \geq \tilde{\hat{S}}_{p}^{l}(\hat{\theta}_{p}^{\rho}) - \tilde{\hat{\lambda}}_{p}^{l}(\hat{\theta}_{p}^{\rho}) (\hat{\theta}_{p} - \hat{\theta}_{p}^{\rho})$,
to all open nodes in $R$\;
}{
Approximate the value function and Lagrange multipliers using
$\tilde{\hat{S}}_{p}^{u}, \tilde{\hat{\lambda}}_{p}^{u}$ for the transition from $x_0$ to $p$ and add Benders cut $\hat{\eta}_p \geq \tilde{\hat{S}}_{p}^{u}(\hat{\theta}_{p}^{\rho}) - \tilde{\hat{\lambda}}_{p}^{u}(\hat{\theta}_{p}^{\rho}) (\hat{\theta}_{p} - \hat{\theta}_{p}^{\rho})$,  
to all open nodes in $R$\;
}
Add MIP-based cuts\;
Update existing nodes in $R$\;}{
Partition domain of $y$ variables\;
Add MIP-based cuts\;
Update existing nodes in $R$\;
}
}
\end{algorithm*}

\section{Numerical case studies} \label{case studies}
In this section we consider numerical case studies to evaluate the computational efficiency of the proposed approach. 

\subsection{Case study 1: 5 products}
We assume that the underlying manufacturing system is an isothermal continuously stirred tank reactor where an irreversible reaction $A \rightarrow 3B$ occurs. The dynamic behavior of the reactor is given by the following differential equation
\begin{equation}
    \frac{dc(t)}{dt} = \frac{F(t)}{V} (c_{feed}-c(t)) - k c(t)^3,
\end{equation}
where $c(t) \geq 0$ is the concentration of A in the reactor, $F(t) \geq 0$ is the inlet flowrate, and $V=5000 L,k = 2 L^2/(hr mol^2),c_{feed}=1 mol/L$ are the volume, reaction constant, and inlet concentration respectively. We assume that five products can be manufactured by changing the inlet flow rate. The monolithic problem has 26400 variables and 26712 constraints for the case where five products are manufactured. The parameters of the optimization problem, such as price, demand, operating conditions, etc., are presented in Table~\ref{table: optimization problem parameters}.

\subsubsection{Learning the surrogate models}
In this section, we present the generation of datasets for learning the surrogate models for the value function and Lagrange multipliers. 

\subsubsection{Transitions between products} \label{subs: trans betw prods}
The value function and Lagrange multipliers for the transitions between the products depend only on the transition time, i.e., $\mathcal{S}_{ijk}(\theta_{ijk})$, $\lambda_{ijk}(\theta_{ijk})$. We learn these functions using supervised learning. For the generation of the training data, for a transition from product $i$ to $j$, we discretize the transition time $\theta_{ijk}$ into $N_{data}=1000$ points between $[\theta_{ij}^{min}, 5 \theta_{ij}^{min}]$. For every value of the transition time we solve the subproblem using IPOPT \cite{wachter2006implementation}, we obtain $\mathcal{S}_{ijk}(\theta_{ijk}^{l})$ and $\lambda_{ijk}(\theta_{ijk}^{l})$ and form two datasets $\{\theta_{ijk}^{l}, \mathcal{S}(\theta_{ijk}^{l})\}_{l=1}^{N_{data}}$, $\{\theta_{ijk}^{l}, \lambda_{ijk}(\theta_{ijk}^{l})\}_{l=1}^{N_{data}}$. 

We use Neural Networks, Decision Trees, and Random Forests to approximate $\mathcal{S}_{ijk}, \lambda_{ijk}$ using Scikit-learn \cite{scikit-learn}. All the neural networks have the same architecture, 3 layers with 50 neurons, and \verb|relu| as an activation function. We use Adam \cite{kingma2014adam} with a learning rate equal to $10^{-4}$ (constant), $L_2$ regularization term equal to $10^{-4}$, the maximum number of iterations is $10^4$, and the solver stops if the objective does not improve for $10^3$ iterations. The Decision Trees and Random Forests are trained using the default values of hyperparameters. 

\subsubsection{Transitions between intermediate state and products}
The value functions for the subproblems that consider the transition from the intermediate state $x_0$ to the steady state of a product depend on the transition time and $x_0$. For every product, we generate a random value for $x_0$ and we compute the minimum transition time $\hat{\theta}_{min}$. Next, we solve multiple dynamic optimization problems for different values of the transition time $\hat{\theta} \in \{ \hat{\theta}_{min}, 2 \hat{\theta}_{min}\}$ and obtain the transition cost. The value functions for product 1 is presented in Fig.~\ref{fig: MIP MIC value functions}.

From this figure, we observe that there is a combination of initial concentration $x_0$ and transition time where the cost is very small. This corresponds to the case where $x_0$ is very close to the steady state concentration of a product $x_{ss}$. Thus the minimum transition time from $x_0$ to $x_{ss}$ and the associated cost are very small. The data presented in Fig.~\ref{fig: MIP MIC value functions} can be used to learn a surrogate model, such as a neural network, $\tilde{\hat{S}}(\hat{\theta},x_0)$ for the transition cost and $\tilde{\hat{\lambda}}(\hat{\theta},x_0)$ for the Lagrange multipliers. However, we observed that the accuracy of these surrogate models is low. To increase the accuracy of the prediction, for every product, we create two surrogate models, one that is used when $x_0>x_{ss}$ and one for $x_0<x_{ss}$. In this approach, every surrogate model approximates a part of the value functions presented in Fig.~\ref{fig: MIP MIC value functions}. 

The generation of the training dataset for the value function and Lagrange multipliers is presented in Algorithm~\ref{alg: ml based bnc benders algorithm sampling mpc}, where for each product first a random value for the concentration in the reactor is generated drawn from a uniform distribution between zero and one. Next, the minimum transition time $\hat{\theta}_{min}$ from the intermediate state to the steady state of product $i$ is computed, and a random value of the transition time between $\{ \hat{\theta}_{min}, 2\hat{\theta}_{min}\}$ is selected uniformly, and the cost and Lagrange multipliers are stored and used to learn the surrogate models. The procedure used to generate the data for learning the surrogate models and the implementation of the ML branch and check the GBD algorithm are presented in Algorithm~\ref{alg: ml based bnc benders algorithm sampling mpc} and Algorithm~\ref{alg: ml based bnc benders algorithm MIP MPC case}.

\begin{table*}[h]
\centering
\caption{Soluton time statistics for different surrogate models for the solution of mixed integer MPC problems for five products. M-GBD and hM-GBD refer to the standard multicut and accelerated hybrid multicut GBD respectively, and NN-GBD, DT-GBD, RF-GBD refer to the implementation of the proposed algorithm using Neural Networks, Decision Trees, and Random Forests as surrogate models.}
\begin{tabular}{cccccccc}
\hline
\multirow{2}{*}{Algorithm} & \multirow{2}{*}{\begin{tabular}[c]{@{}c@{}}Average\\ sol. time (sec.)\end{tabular}} & \multirow{2}{*}{\begin{tabular}[c]{@{}c@{}}Standard \\ deviation\end{tabular}} & \multirow{2}{*}{\begin{tabular}[c]{@{}c@{}}Average\\ percentage error\end{tabular}} & \multicolumn{2}{c}{Percentage reduction} & \multicolumn{2}{c}{Average fold reduction} \\ \cline{5-8}  &  &    &  & M-GBD  & hM-GBD  & M-GBD   & hM-GBD  \\ \hline
M-GBD    & 27.31  & 5.19 & - & - & - & - & -\\
hM-GBD   & 19.01  & 1.96 & - & -   & - & -& -\\ \hline
NN-GBD& 0.83& 0.29& 0.23& 96.98& 95.66& 35& 25\\
DT-GBD& \textbf{0.55}& \textbf{0.15}& \textbf{0.21}& \textbf{97.97}& \textbf{97.09}& \textbf{50}& \textbf{36}\\
RF-GBD& 4.00& 1.33& 0.21& 85.47& 78.88& 7.3& 5\\ \hline
\end{tabular}
\label{table:ML MIP MPC sol time state}
\end{table*}

\subsubsection{Numerical results}
We generate 20 feasible random disturbances that were not used for training to test the efficiency of the proposed approach. The proposed approach is implemented in Pyomo \cite{hart2017pyomo} using callbacks in Gurobi \cite{gurobi} with \verb|LazyCuts|. We compare the efficiency of the proposed approach with two algorithms from the literature. For both algorithms (from the literature), the decomposition of the problem is the same as the one in this paper. The first algorithm is a standard multicut GBD algorithm where in each iteration, the cuts are added only for the transitions that occur. The second algorithm, proposed in \cite{mitrai2022multicut}, is an accelerated hybrid multicut GBD algorithm, where the cuts are computed for a given transition $(i,j,k)$ but are added for all $(i,j,k')$ where $k' \in \mathcal{I}_s$. We refer to the standard multicut algorithm as M-GBD and the hybrid multicut as hM-GBD. The application of the hybrid multicut algorithms can be found in \cite{mitrai2023focapo}. The solution time statistics of the different algorithms are presented in Table~\ref{table:ML MIP MPC sol time state}.

The average solution time of standard multicut GBD and hybrid multicut GBD is $27.31$ and $19.01$ seconds, respectively. The solution time with the proposed approach is $0.83$, $0.55$, and $4$ seconds when the Neural Networks, Decision Trees, and Random Forests are used as surrogate models for the transition cost and Lagrangean multipliers. These results show that the proposed approach can lead to a significant reduction in solution time; the average reduction in solution time is $97 \% \ (50 \times)$ and $95 \% \ (36\times)$ when the neural networks and decision trees are used as surrogate models. 

Next, we compare the standard deviation of the different solution approaches. From Table~\ref{table:ML MIP MPC sol time state} we observe that the standard deviation of the solution time with the proposed approach is lower than the multicut and hybrid multicut algorithms. Specifically, the standard deviation of the standard multicut algorithm is $5.19$ seconds and for the hybrid multicut GBD algorithm is $1.96$ seconds, whereas the standard deviation of the proposed approach is $0.29$, $0.15$, and $1.33$ seconds when the Neural Networks, Decision Trees and Random Forests are used as surrogate models. These results show that the proposed approach is more robust than both the standard and accelerated multicut GBD algorithms.

Also, we analyze the error in the optimal solution that is obtained using the multicut and hybrid multicut GBD algorithms from \cite{mitrai2022multicut} and the proposed approach. The error is computed as follows
\begin{equation}
    Error_i = 100 \times \frac{|f^*_{GBD} - f^*_{i-GBD}|}{f^*_{GBD}} \ \forall i=\{NN,DT,RF\}.
\end{equation}
In the above equation, $f^*_{GBD}$ is the optimal solution obtained either via the multicut or hybrid multicut GBD since both these algorithms converge to the same solution (see \cite{mitrai2022multicut}). The average percentage error is presented in Table~\ref{table:ML MIP MPC sol time state}. We observe that all the surrogate models have similar average errors, and the decision tree shows the minimum average error of $0.21 \%$. Overall, these results show that the proposed approach can find high-quality feasible solutions with small errors in significantly reduced solution time, characteristics that are necessary for the online solution of mixed integer economic MPC problems encountered in industrial applications.

\subsection{Case study 2: Effect of number of products on computational performance}
In the second case study, we analyze the effect of the number of products on the computational efficiency of the proposed approach. We solve the problem for five up to eight products. The surrogate models are learned using the steps from subsection~\ref{subs: trans betw prods} and Algorithm~\ref{alg: ml based bnc benders algorithm sampling mpc}. The values of the parameter of the optimization problems are presented in the Supplementary Material. For a given number of products, we generate 20 random feasible disturbances, and the solution time statistics for different numbers of products are presented in Table~\ref{table: solt time stat for diff num of prods} and the average solution time is presented in Fig.~\ref{fig:MIP MPC sol time multiple prods}.
\begin{figure}[h]
    \centering
    \includegraphics[scale=0.6]{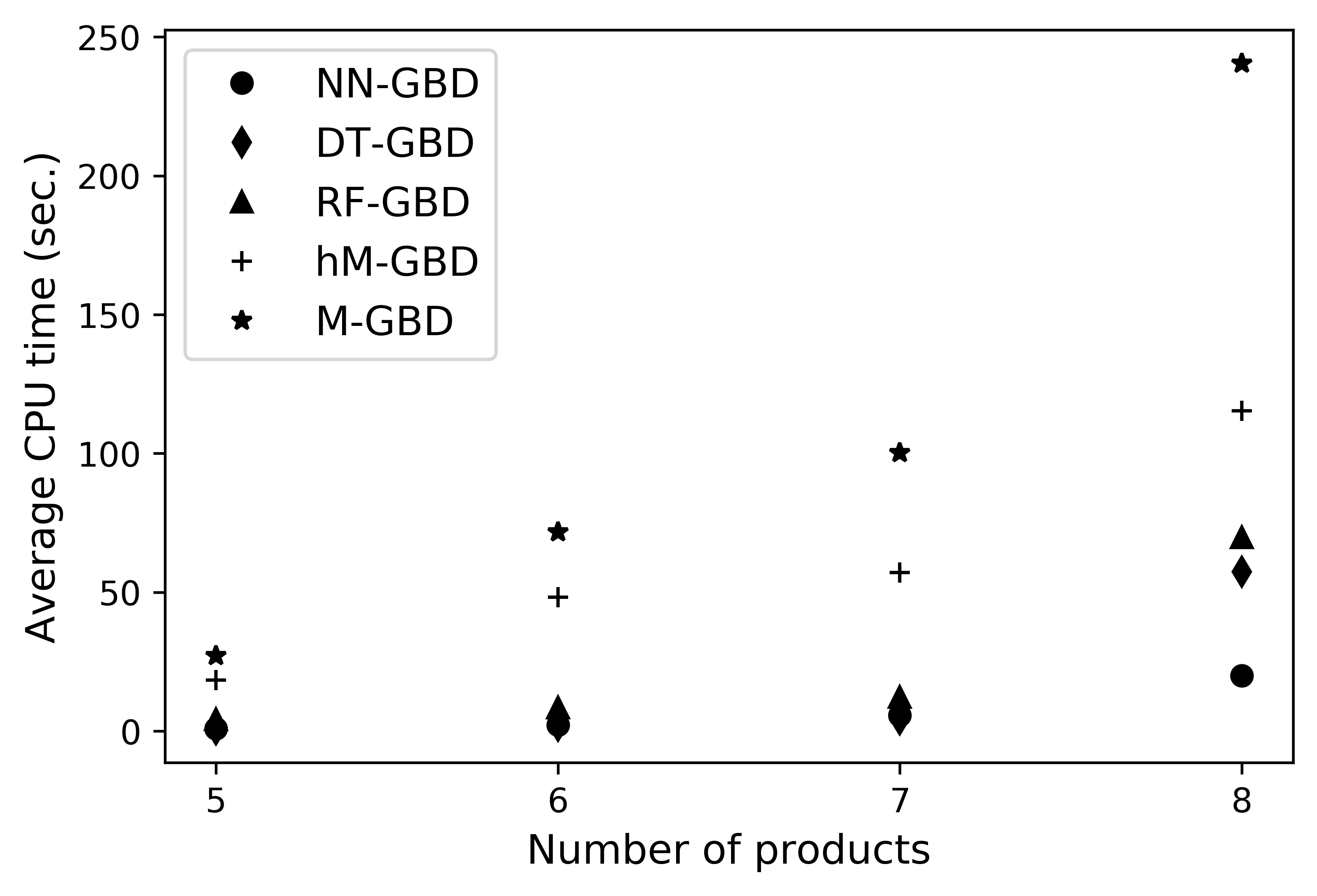}
    \caption{Solution time for 20 randomly generated instances using the multicut GBD (M-GBD), hybrid multicut GBD (hM-GBG) from \cite{mitrai2022multicut} and the proposed method using different surrogate models, Neural Network (NN-GBD), Decision Tree (DT-GBD), and Random Forest (RF-GBD) for the different number of products.}
    \label{fig:MIP MPC sol time multiple prods}
\end{figure}

From the results, we observe that an increase in the number of products leads to an increase in the solution time for all the algorithms. Specifically, for the multicut GBD algorithm, the solution time increases from $27.31$ seconds for five products to $71.70$ seconds for six, $100$ seconds for seven, and $240$ seconds for eight products. Similarly, for the hybrid multicut algorithm, the average solution time increases from $19$ seconds for five products to $48.21$ seconds for six, $57.10$ seconds for seven, and $115.27$ seconds for eight products. Regarding the proposed approach, although the average solution time increases, the average solution time is significantly lower than the solution time of the multicut and hybrid multicut GBD algorithms. Specifically, for the Neural Network surrogate models, the solution time increases from $0.83$ seconds for 5 products to $2.14$ seconds for 6 seconds, $5.72$ for 7 seconds, and $19.85$ for 8 products. A similar increase is observed for the Decision Tree and Random Forest surrogate models. Comparing the effect of the different surrogate models, we observe that for five, six, and seven products, the Decision Trees lead to lower CPU time and error. For eight products, the neural networks have the minimum CPU time, and the Decision Trees have the lowest average error. From these results, we observe that the proposed approach is significantly faster than standard versions of GBD for all the number of products considered.

Furthermore, from the numerical results, we observe that the standard deviation of the solution time for the proposed approach is lower than the standard deviation of the multicut and hybrid muticut algorithms. The variance in the solution time is caused by the change in the complexity of the optimization problem that is solved once a disturbance affects the systems. Since the disturbance occurs at a random time point $T_0$, the number of products that must be manufactured can vary based on the demand and the amount produced for every product up to time $T_0$. However, the standard deviation of the proposed approach is lower than the standard deviation of the multicut and hybrid multicut GBD algorithms as reported in Table~\ref{table: solt time stat for diff num of prods}. Finally, we analyze the error in the optimal solution obtained with the proposed approach. From the results in Table~\ref{table: solt time stat for diff num of prods} we observe that the average error is in the order of $1 \%$ for the Neural Networks and $10^{-1}$ for the Decision Tree and the Random Forest. 

\begin{table}[h]
\centering
\caption{Solution time statistics for the multicut (M-GBD) and hybrid multicut (hM-GBD) GBD algorithms from \cite{mitrai2022multicut} and the proposed method using different surrogate models, Neural Network (NN-GBD), Decision Tree (DT-GBD), and Random Forest (RF-GBD), for different number of products.}
  \resizebox{\columnwidth}{!}{  
\begin{tabular}{cccccc}
\hline
\multirow{2}{*}{\begin{tabular}[c]{@{}c@{}}Solution time\\ statistics\end{tabular}}                             & \multirow{2}{*}{Algorithm} & \multicolumn{4}{c}{Number of products} \\ \cline{3-6} 
&                            & 5       & 6       & 7        & 8       \\ \hline
\multirow{5}{*}{\begin{tabular}[c]{@{}c@{}}Average \\ solution\\ time\end{tabular}}                             & M-GBD                      & 27.31   & 71.70   & 100.18   & 240.47  \\
& hM-GBD  & 19.01   & 48.21   & 57.10    & 115.27  \\
& NN-GBD  & 0.83    & 2.14    & 5.72     & \textbf{19.85}   \\
& DT-GBD  & \textbf{0.55}    & \textbf{1.84}    & \textbf{4.26}     & 57.30   \\
& RF-GBD    & 4.00    & 8.42    & 23.49    & 69.70   \\ \hline
\multirow{5}{*}{\begin{tabular}[c]{@{}c@{}}Standard \\ deviation\end{tabular}}                                  
& M-GBD  & 5.19    & 39.51   & 32.08    & 111.29  \\
& hM-GBD  & 1.96    & 25.95   & 19.55    & 71.09   \\
& NN-GBD & 0.29    & 1.07    & 3.07     & 9.11    \\
& DT-GBD  & 0.13    & 1.21    & 2.89     & 35.37   \\
& RF-GBD  & 1.33    & 3.18    & 4.28     & 39.93   \\ \hline
\multirow{3}{*}{\begin{tabular}[c]{@{}c@{}}Average\\ error\end{tabular}}                                        
& NN-GBD   & 0.23    & 1.39    & 1.15     & 1.52    \\
& DT-GBD  & \textbf{0.21}    & \textbf{0.39}   & \textbf{0.19}     & \textbf{0.18}    \\
& RF-GBD  & 0.21    & 0.37    & 0.20     & 0.18    \\ \hline
\multirow{3}{*}{\begin{tabular}[c]{@{}c@{}}Average sol. time reduction\\ compared to M-GBD\end{tabular}}  
& NN-GBD     & \textbf{97}      & 96      & 94       & \textbf{91}      \\
& DT-GBD  & \textbf{97}      & \textbf{97}      & \textbf{95}       & 75      \\
& RF-GBD   & 85      & 80      & 87       & 69      \\ \hline
\multirow{3}{*}{\begin{tabular}[c]{@{}c@{}}Average sol. time reduction \\compared to hM-GBD\end{tabular}} 
& NN-GBD   & 95      & \textbf{95}      & 89       & \textbf{80}      \\
& DT-GBD  & \textbf{97} & \textbf{95}      & \textbf{92}       & 46      \\
& RF-GBD  & 78      & 86      & 77       & 32      \\ \hline
\multirow{3}{*}{\begin{tabular}[c]{@{}c@{}}Average fold reduction\\ compared to M-GBD\end{tabular}}         & NN-GBD                     & 35      & 37      & 23       & \textbf{13}      \\
& DT-GBD   & \textbf{50}      & \textbf{47}      & \textbf{31}       & 8       \\
& RF-GBD & 7       & 8       & 8        & 4       \\ \hline
\multirow{3}{*}{\begin{tabular}[c]{@{}c@{}}Average fold reduction\\ compared to hM-GBD\end{tabular}}       
& NN-GBD  & 25      & 26      & 13       & \textbf{6}       \\
& DT-GBD  & \textbf{36}& \textbf{34}      & \textbf{18}       & 4       \\
& RF-GBD  & 5       & 6       & 5        & 2       \\ \hline
\end{tabular}}
\label{table: solt time stat for diff num of prods}
\end{table}

\section{Conclusions}
Mixed integer MPC problems arise in a wide range of applications. The implementation of such a control strategy depends on the efficient solution of a mixed integer optimization problem. In this paper, we combine decomposition-based optimization algorithms with machine learning to accelerate the online solution of mixed integer MPC problems. Specifically, we propose a machine learning-based branch and check GBD algorithm, where the information required for the construction of Benders cuts is approximated via ML-based surrogate models, which are learned offline. Application to mixed integer MPC of chemical processes shows that the proposed approach can lead up to $97 \%$ reduction in solution time while incurring small error (in the order of $1 \%$) in the value of the objective function. Finally, we point out a few extensions of the proposed approach in the remarks below. 

\begin{remark}
    \normalfont As discussed throughout the paper, the proposed approach is based on approximating the value function of the subproblem and the Lagrange multipliers. However, obtaining a large dataset, even offline, can be computationally challenging. Consider, for example, the case where the subproblem is a large-scale optimization problem or a nonlinear optimization problem that requires an initial guess for the optimal solution. In such cases, the data generation process can be nontrivial. For example, one can use efficient sampling techniques based on active learning, semi-supervised techniques to learn from a limited set of data, or transfer learning. 
\end{remark}
\begin{remark}
    \normalfont In the case studies considered in this paper, the original problem was decomposed into a master problem and a set of subproblems, and the same hyperparameters were used for all the surrogate models and transitions. However, one can potentially use different hyperparameters and surrogate models for the different subproblems, which can be obtained via standard hyperparameter tuning techniques. 
\end{remark}
\begin{remark}
   \normalfont As discussed in Section~\ref{feas and opt of mip mpc}, the solution returned by the proposed approach is feasible if the original problem is feasible and the subproblem is feasible for all the values of the complicating variables. However, for general mixed integer MPC problems, the subproblem can be infeasible for certain values of the complicating variables. In such cases, the proposed approach can still be applied but must be combined with an infeasibility detection and a feasibility restoration step. The infeasibility detection step will predict whether the subproblem is feasible for given values of the complicating variables, and based on this prediction, Benders feasibility or optimality cuts will be approximated and added to the master problem. The feasibility restoration phase is necessary to guarantee that the solution that is obtained is feasible. This restoration step is common in many hybrid optimization algorithms that combine machine learning and mathematical optimization \cite{chen2023end,chatzos2021spatial}.
\end{remark}

\section{Acknowledgment}
Financial support from NSF-CBET is gratefully acknowledged.

\bibliographystyle{elsarticle-num}
\bibliography{sample}

\end{document}